\documentclass{article}

\usepackage{amssymb,amsmath,mathtools,amsthm}

\usepackage[figuresright]{rotating}

\usepackage[colorlinks,linkcolor=red,citecolor=red]{hyperref}
\newtheorem{thm}{Theorem}[section]
\newtheorem{cor}[thm]{Corollary}
\newtheorem{lem}[thm]{Lemma}
\newtheorem{prop}[thm]{Proposition}

\begin{document}

%\begin{frontmatter}

\title{Status connectivity indices and co-indices of graphs and its computation to intersection graph, hypercube, Kneser graph and achiral polyhex nanotorus}

\author{Harishchandra S. Ramane$^a$\footnote{hsramane@yahoo.com},\quad Ashwini S. Yalnaik$^a$\footnote{ashwiniynaik@gmail.com}, \quad Reza Sharafdini$^b$\footnote{sharafdini@pgu.ac.ir}\\[3mm]
$^a$Department of Mathematics, Karnatak University, Dharwad-580003, India\\
$^b$ Department of Mathematics, Faculty of Science, Persian Gulf University,\\
 Bushehr, 7516913817, Iran
}

%\address{}
%\author[rvt]{Harishchandra S. Ramane\corref{cor1}}
%\ead{hsramane@yahoo.com}
%\author[rvt]{Ashwini S. Yalnaik}
%\ead{ashwiniynaik@gmail.com}
%\author[rsh]{Reza Sharafdini}
%\ead{sharafdini@pgu.ac.ir}
%\address[rvt]{Department of Mathematics, Karnatak University, Dharwad - 580003, India}
%\address[rsh]{Department of Mathematics, Faculty of Science, Persian Gulf University, Bushehr, 7516913817, Iran}
%\begin{keyword}

%\MSC[2010] 05C12, 92E10.
%\end{keyword}

%\end{frontmatter}
\maketitle

\begin{abstract}
The status of a vertex $u$ in a connected graph $G$, denoted by $\sigma_G(u)$, is defined as the sum of the distances between $u$ and all other vertices of a graph $G$. The first and second status connectivity indices of a graph $G$ are defined as $S_{1}(G) = \sum_{uv \in E(G)}[\sigma_G(u)+ \sigma_G(v)]$ and $S_{2}(G) = \sum_{uv \in E(G)}\sigma_G(u)\sigma_G(v)$ respectively, where $E(G)$ denotes the edge set of $G$. In this paper we have defined the first and second status co-indices of a graph $G$ as $\overline{S_{1}}(G) = \sum_{uv \notin E(G)}[\sigma_G(u)+ \sigma_G(v)]$ and $\overline{S_{2}}(G) = \sum_{uv \notin E(G)}\sigma_G(u)\sigma_G(v)$ respectively. Relations between status connectivity indices and status coindices are established. Also these indices are computed for intersection graph, hypercube, Kneser graph and achiral polyhex nanotorus.\\

\end{abstract}

\section{Introduction}
The graph theoretic models can be used to study the properties of molecules in theoretical chemistry. The oldest well known graph parameter is the Wiener index which was used  to study the chemical properties of paraffins \cite{Wie}. The Zagreb indices were used to study the structural property models \cite{Gut4,Tod}. Recently, Ramane and Yalnaik \cite{Ram1}, introduced the status connectivity indices based on the distances and correlated it with the boiling point of benzenoid hydrocarbons. In this paper we define the status co-indices of a graph and establish the relations between the status connectivity indices and status co-indices. Also we obtain the bounds for the status conncitivity indices of connected complement graphs. Furher we compute these status indices for intersection graph, hypercube, Kneser graph and achiral polyhex nanotorus.\\

Let $G$ be a connected graph with $n$ vertices and $m$ edges. Let $V(G)= \{v_{1},v_{2}, \ldots, v_{n}\}$ be the vertex set of $G$ and $E(G)$ be an edge set of $G$. The edge joining the vertices $u$ and $v$ is denoted by $uv$. The \textit{degree} of a vertx $u$ in a graph $G$ is the number of edges joining to $u$ and is denoted by $d_{G}(u)$. The \textit{distance} between the vertices $u$ and $v$ is the length of the shortest path joining $u$ and $v$ and is denoted by $d_{G}(u,v)$.\\

The \textit{status} (or \textit{transmission})  of a vertex $u \in V(G)$, denoted by $\sigma_G(u)$ is defined as \cite{Har1},
\[\sigma_G(u) = \sum_{v\in V(G)}d(u,v).\]

A connected graph $G$ is said to be \textit{$k$-transmission regular} if $\sigma_G(u) = k$ for every vertex $u\in V(G)$. The transmission regular graphs are exactly the \textit{distance-balanced} graphs introduced in \cite{Handa}. They are
also called as \textit{self-median} graphs \cite{Cabello}.

The \textit{Wiener index} $W(G)$ of a connected graph $G$ is defined as \cite{Wie},
\[ W(G) = \sum_{\{u, v\} \subseteq V(G)}d_G(u,v) = \frac{1}{2}\sum_{u \in V(G)} \sigma_G(u). \]

More results about Wiener index can be found in \cite{Das1,Dob,Gut5,Nik,Ram2,Ram3,Wal}.\\

The first and second \textit{Zagreb indices} of a graph $G$ are defined as \cite{Gut4}
\[ M_1(G) = \sum_{uv \in E(G)}\left[d_G(u) + d_G(v) \right] \hspace{5mm} \mathrm{and} \hspace{5mm} M_2(G) = \sum_{uv \in E(G)}d_G(u)d_G(v). \]

Results on the Zagreb indices can be found in \cite{Das2,Gut1,Gut2,Kha,Nik2,Zhou}.\\

The first and second Zagreb co-indices of a graph $G$ are defined as \cite{Dos1}
\[ \overline{M_{1}}(G)= \sum_{uv \not\in E(G)}\left[d_{G}(u) + d_{G}(v) \right] \hspace{5mm} \mathrm{and} \hspace{5mm}  \overline{M_{2}}(G)= \sum_{uv \not\in E(G)}\left[d_{G}(u)  d_{G}(v) \right]. \]

More results on Zagreb coindices can be found in \cite{Ash3,Ash4}.

Recently, the first and second status connectivity index of a graph $G$ have been introduced by Ramane and Yalnaik \cite{Ram1} to study the property of benzenoid hydrocarbons and these are defined as
\begin{equation} \label{Eq1}
S_1(G) = \sum_{uv \in E(G)}\left[\sigma_G(u) + \sigma_G(v) \right]
\hspace{5mm} \mathrm{and} \hspace{5mm}
S_2(G) = \sum_{uv \in E(G)}\sigma_G(u)\sigma_G(v).
\end{equation}

Similar to Eq. (\ref{Eq1}) and the definition of Zagreb co-index, we define here the first status co-index $\overline{S_{1}}(G)$ and the second status co-index $\overline{S_{2}}(G)$ as

\[ \overline{S_{1}}(G)= \sum_{uv \not\in E(G)}\left[\sigma_G(u) + \sigma_G(v) \right]
\hspace{5mm} \mathrm{and} \hspace{5mm}
\overline{S_{2}}(G)= \sum_{uv \not\in E(G)}\left[\sigma_G(u)  \sigma_G(v) \right]. \]

\noindent
\textbf{Example:}
\begin{center}
	\begin{picture}(65,75)(4,10)
	\thicklines
	\put(20,20){\line(1,0){40}} \put(20,60){\line(1,0){40}} \put(20,20){\line(1,1){40}}
	\put(20,60){\line(1,1){20}} \put(60,60){\line(-1,1){20}}
	\put(20,20){\line(0,1){40}}  \put(60,20){\line(0,1){40}} \put(40,82){\line(-1,-3){20}}
	
	\put(60,20){\circle*{5}} \put(20,60){\circle*{5}} \put(20,20){\circle*{5}} \put(60,60){\circle*{5}} \put(40,80){\circle*{5}}
	
	\put(8,20){$v_3$} \put(64,20){$v_4$} \put(8,60){$v_2$} \put(35,84){$v_1$} \put(64,60){$v_5$}
	
	\put(20,3){Figure 1}
		
	\end{picture}
	\end{center}
\vspace{2mm}
For a graph given in Fig. 1, $S_{1}=74$, $S_{2}= 169$, $\overline{S_{1}}=11$, $\overline{S_{2}}= 60$.\\

\section{Status connectivity indices and co-indices}
\noindent
Status connectivity indices of connected graphs are obtained in [24], In this section we obtain the status coindices and also status indices of complements.\\

 \begin{prop} \label{Prop2.1}
	Let $G$ be a connected graph  on $n$ vertices.\\
	Then
	\[\overline{S_{1}}(G)= 2(n-1)W(G)- S_{1}(G)\]
	and
	\[ \overline{S_{2}}(G) = 2(W(G))^2 - \frac{1}{2}\sum_{u \in V(G)}(\sigma_G(u))^2 - S_{2}(G).\]
\end{prop}
\begin{proof}
	\begin{eqnarray}\nonumber
	\overline{S_{1}}(G)& = & \sum_{uv \not \in E(G)} \left[\sigma_{G}(u)+\sigma_{G}(v)\right]\\ \nonumber
	& = & \sum_{\left\lbrace u,v \right\rbrace \subseteq V(G)} \left[\sigma_{G}(u)+\sigma_{G}(v)\right]- \sum_{uv  \in E(G)} \left[\sigma_{G}(u)+\sigma_{G}(v)\right]\\ \nonumber
	& = & (n-1)\sum_{u \in V(G)} \sigma_{G}(u) - S_{1}(G)\\ \nonumber
	& = & 2(n-1)W(G)- S_{1}(G).
	\end{eqnarray}
	
	Also
	
	\begin{eqnarray}\nonumber
	\overline{S_{2}}(G)& = & \sum_{uv \not\in E(G)}\left[\sigma_G(u)  \sigma_G(v) \right] \\ \nonumber
	& = & \sum_{\left\lbrace u,v \right\rbrace \subseteq V(G)} \left[\sigma_{G}(u)\sigma_{G}(v)\right] - \sum_{uv \in E(G)}\left[\sigma_G(u) \sigma_G(v) \right] \\ \nonumber
	& = & \frac{1}{2}\left[\left[\sum_{u \in V(G)} \sigma_G(u)\right]^2 - \sum_{u \in V(G)} \sigma_G(u)^{2}\right] - S_{2}(G)\\\nonumber
	& = & 2(W(G))^2 - \frac{1}{2}\sum_{u \in V(G)}(\sigma_G(u))^2 - S_{2}(G).
	\end{eqnarray}
\end{proof}

\begin{cor} \label{Cor2.2}
	Let $G$ be a connected graph with $n$ vertices, $m$ edges and $diam(G)\leq 2$. Then
	\[ \overline{S_{1}}(G) = 2n(n-1)^2 -6m(n-1)+ M_{1}(G) \]
	and
	\[\overline{S_{2}}(G) =  (n-1)^2\left[2n(n-1)-8m\right] + 2m^2 + \left(2n - \frac{5}{2} \right)M_1(G) - M_2(G).\]
\end{cor}
\begin{proof}
	For any graph $G$ of $diam(G) \leq 2$, $\sigma_G(u)=2n-2-d_G(u)$ and
	\[ W(G) = m+ 2\left[\frac{n(n-1)}{2}-m \right] = n(n-1)-m. \]
	Also $S_1(G) = 4m(n - 1) - M_1(G)$ and $S_2(G) = 4m(n - 1)^2 - 2(n - 1)M_1(G) + M_2(G)$ \cite{Ram1}.\\
	Therefore by Proposition \ref{Prop2.1},
	\begin{eqnarray}\nonumber
	\overline{S_{1}}(G) & = & 2(n-1)[n(n-1)-m]- \left\lbrace 4m(n-1)- M_{1}(G)\right\rbrace\\\nonumber
	& = & 2n(n-1)^2- 6m(n-1)+M_{1}(G)
	\end{eqnarray}
	and
	\begin{eqnarray}\nonumber
	\overline{S_{2}}(G) & = & 2\left[n(n-1) - m\right]^2 - \frac{1}{2}\sum_{u \in V(G)}(2n-2-d_G(u))^2 \\ \nonumber
	& & - \; \left[4m(n-1)^2 - 2(n-1)M_1(G) + M_2(G) \right]\\\nonumber
	& = & 2\left[n(n-1) - m\right]^2 - \frac{1}{2}\left[\sum_{u \in V(G)}(2n-2)^2 - 2(2n-2)\sum_{u \in V(G)}d_G(u) \right.\\ \nonumber
	& & \left. + \; \sum_{u \in V(G)}(d_G(u))^2 \right] - \left[4m(n-1)^2 - 2(n-1)M_1(G) + M_2(G) \right]\\\nonumber
	& = & 2\left[n(n-1) - m\right]^2 - \frac{1}{2}\left[n(2n-2)^2 - 4m(2n-2) + M_1(G) \right] \\ \nonumber
	& & - \; \left[4m(n-1)^2 - 2(n-1)M_1(G) + M_2(G) \right]\\\nonumber
	& = & (n-1)^2\left[2n(n-1)-8m\right] + 2m^2 + \left(2n - \frac{5}{2} \right)M_1(G) - M_2(G).
	\end{eqnarray}
\end{proof}

\begin{prop} \label{Prop2.3}
	Let $G$ be a connected graph with $n$ vertices, $m$ edges and $diam(G)\leq 2$. Then
	\[ \overline{S_{1}}(G) = 2(n-1)\left[n(n-1) - 2m \right] - \overline{M_{1}}(G) \]
	and
	\[\overline{S_{2}}(G) =  2(n-1)^2\left[n(n-1)-2m\right] -2(n-1)\overline{M_{1}}(G)+\overline{M_{2}}(G). \]
\end{prop}
\begin{proof}
	For any graph $G$ of $diam(G) \leq 2$, $\sigma_G(u)=2n-2-d_G(u)$.
	Therefore
	\begin{eqnarray}\nonumber
	\overline{S_{1}}(G) & = & \sum_{uv \notin E(G)}\left[ (2n-2-d_G(u)) + (2n-2-d_G(v))\right]\\\nonumber
	&  = & \left[ \frac{n(n-1)}{2}-m \right](4n-4) -\sum_{uv \notin E(G)}\left[d_G(u)+d_G(v)\right] \\ \nonumber
	& = & 2(n-1)\left[n(n-1) - 2m \right] - \overline{M_{1}}(G).
	\end{eqnarray}
	and
	\begin{eqnarray}\nonumber
	\overline{S_{2}}(G) & = & \sum_{uv \not \in E(G)}\left[ (2n-2-d_G(u))(2n-2-d_G(v))\right]\\\nonumber
	&  = & \left[ \frac{n(n-1)}{2}-m \right](2n-2)^2 -(2n-2)\sum_{uv \not \in E(G)}\left[d_G(u)+d_G(v)\right] \\ \nonumber
	& & + \; \sum_{uv \not \in E(G)}(d_G(u)d_G(v))\\\nonumber
	& = & 2(n-1)^2\left[ n(n-1)-2m\right]-2(n-1)\overline{M_{1}}(G)+\overline{M_{2}}(G).
	\end{eqnarray}
\end{proof}

\begin{prop} \label{Prop2.4}
	Let $G$ be a graph with $n$ vertices and $m$ edges. Let $\overline{G}$, the complement of $G$, be connected. Then
	\[ S_{1}(\overline{G}) \geq (n-1)[n(n-1)-2m]+ \overline{M_{1}}(G) \]
	and
	\[ S_{2}(\overline{G}) \geq  (n-1)^2 \left[\frac{n(n-1)}{2}-m \right] + (n-1)\overline{M_{1}}(G)+\overline{M_{2}}(G). \]
	Equality holds if and only if $diam(\overline{G}) \leq 2$.
\end{prop}
\begin{proof}
	For any vertex $u$ in $\overline{G}$ there are $n-1-d_{G}(u)$ vertices which are at distance $1$ and the remaining $d_{G}(u)$ vertices are at distance at least $2$. Therefore\\
	\begin{eqnarray}\nonumber
	\sigma_{\overline{G}}(u) & \geq & [n-1+ d_{G}(u)]+2 d_{G}(u)\\ \nonumber
	& = & n-1+d_{G}(u).
	\end{eqnarray}
	Therefore,
	\begin{eqnarray}\nonumber
	S_{1}(\overline{G})& = & \sum_{uv \in E(\overline{G})}\left[ \sigma_{\overline{G}}(u)+\sigma_{\overline{G}}(v)\right]\\\nonumber
	& \geq & \sum_{uv \in E(\overline{G})}\left[  n-1+d_{G}(u)+ n-1+d_{G}(v)\right]\\ \nonumber
	& = & \sum_{uv \not \in E(G)}\left[2n-2+d_{G}(u)+d_{G}(v)\right]\\\nonumber
	& = & \left[ \frac{n(n-1)}{2}-m\right](2n-2)+ \sum_{uv \not \in E(G)}[d_{G}(u)+d_{G}(v)]\\\nonumber
	& = & [n(n-1)-2m](n-1)+\overline{M_{1}}(G).
	\end{eqnarray}
	And
	\begin{eqnarray}\nonumber
	S_{2}(\overline{G}) & = & \sum_{uv \in E(\overline{G})}\sigma_{\overline{G}}(u)\sigma_{\overline{G}}(v)\\\nonumber
	& \geq & \sum_{uv \in E(\overline{G})}[n-1+d_{G}(u)][n-1+d_{G}(v)]\\\nonumber
	& = & \sum_{uv \not \in E(G)}\left[(n-1)^2+(n-1)[d_{G}(u)+d_{G}(v)]+[d_{G}(u)d_{G}(v)]\right]\\\nonumber
	& = & \left[\frac{n(n-1)}{2}-m \right](n-1)^2+(n-1)\overline{M_{1}}(G)+\overline{M_{2}}(G).
	\end{eqnarray}
	
	For equality: If the diameter of $\overline{G}$ is $1$ or $2$ then the equality holds.
	
	Conversely, let $S_{1}(\overline{G}) = (n-1)[n(n-1)-2m]+ \overline{M_{1}}(G)$.
	
	Suppose, $diam(\overline{G}) \geq 3$, then there exists at least one pair of vertices, say $u_1$ and $u_2$ such that $d_{\overline{G}}(u_1, u_2) \geq 3$. Therefore $\sigma_{\overline{G}}(u_1) \geq d_{\overline{G}}(u_1) + 3 + 2(n-2-d_{\overline{G}}(u_1)) = n + d_G(u_1)$. Similarly $\sigma_{\overline{G}}(u_2) \geq n + d_G(u_2)$ and for all other vertices $u$ of $\overline{G}$, $\sigma_{\overline{G}}(u) \geq n - 1 +d_G(u)$.
	
	Partition the edge set of $\overline{G}$ into three sets $E_1$, $E_2$ and $E_3$, where
	$E_1 = \{ u_1v \,\,| \,\, \sigma_{\overline{G}}(u_1) \geq n + d_G(u_1) \,\, \text{and} \, \, \sigma_{\overline{G}}(v) \geq n -1 +d_G(v) \}$, $E_2 = \{ u_2v \,\,| \,\, \sigma_{\overline{G}}(u_2) \geq n + d_G(u_2) \,\, \text{and} \, \, \sigma_{\overline{G}}(v) \geq n -1 +d_G(v) \}$ and $E_3 = \{ uv \,\,| \,\, \sigma_{\overline{G}}(u) \geq n - 1 + d_G(u) \,\, \text{and} \, \, \sigma_{\overline{G}}(v) \geq n -1 +d_G(v) \}$. It is easy to check that $|E_1| = d_{\overline{G}}(u_1)$, $|E_2| = d_{\overline{G}}(u_2)$ and $|E_3| = \binom{n}{2} - m - d_{\overline{G}}(u_1) -d_{\overline{G}}(u_2) $.
	
	Therefore
	\begin{eqnarray}\nonumber
	S_{1}(\overline{G})& = & \sum_{uv \in E(\overline{G})}\left[ \sigma_{\overline{G}}(u)+\sigma_{\overline{G}}(v)\right]\\\nonumber
	& = & \sum_{uv \in E_1}\left[ \sigma_{\overline{G}}(u)+\sigma_{\overline{G}}(v)\right] + \sum_{uv \in E_2}\left[ \sigma_{\overline{G}}(u)+\sigma_{\overline{G}}(v)\right] + \sum_{uv \in E_3}\left[ \sigma_{\overline{G}}(u)+\sigma_{\overline{G}}(v)\right] \\ \nonumber
	& = & \sum_{uv \in E_1}\left[2n - 1 + d_G(u) + d_G(v) \right] + \sum_{uv \in E_2}\left[2n - 1 + d_G(u) + d_G(v) \right] \\ \nonumber
	& & +\, \sum_{uv \in E_3}\left[ 2n - 2 + d_G(u) + d_G(v) \right] \\ \nonumber
	& = & (2n-1)d_{\overline{G}}(u_1) + (2n-1)d_{\overline{G}}(u_2) \\ \nonumber
	& & +\, (2n-2)\left[\binom{n}{2} - m - d_{\overline{G}}(u_1) - d_{\overline{G}}(u_2) \right] + \sum_{uv \in E(\overline{G})}\left[d_G(u) + d_G(v) \right] \\ \nonumber
	& = & (n-1)[n(n-1)-2m]+d_{\overline{G}}(u_1) + d_{\overline{G}}(u_2) + \overline{M_{1}}(G),
	\end{eqnarray}
	which is a contradiction. Hence $diam(\overline{G}) \leq 2$.
	
	The equality of $S_2(\overline{G})$ can be proved analogously.
\end{proof}

\section{Status indices and co-indices of some transmission regular graphs}
Status indices of some standard graphs are obtained in \cite{Ram1}.\\

A bijection  $ \alpha $  on $V(G)$ is called an \textit{automorphism} of $G$ if it preserves $E(G)$. In other words, $\alpha$ is an automorphism if  for each $u,v\in V(G)$, $e = uv\in E(G)$ if and only if
$ \alpha (e)=\alpha (u)\alpha (v) \in E(G)$. Let
\[Aut(G)=\{\alpha \mid \alpha:V(G)\to V(G)~~\text{is a bijection, which preserves the adjacency}\}.\]
It is known that $ Aut(G) $  forms a group under the composition of mappings.
A graph $G$ is called \textit{vertex-transitive} if  for
every two vertices $u$ and $v$ of $G$, there exists an automorphism $\alpha$ of $G$ such that $\alpha( u ) = \alpha( v )$. It is known that any vertex-transitive graph is vertex degree regular, transmission regular and self-centred.  Indeed, the graph depicted in Figure 2 is 14-transmission regular graph but
not degree regular and therefore not vertex-transitive (see \cite{Aouchiche2001, Aouchiche2010}).

\begin{figure}[h!]
	\centering
	\includegraphics[width=2.5cm]{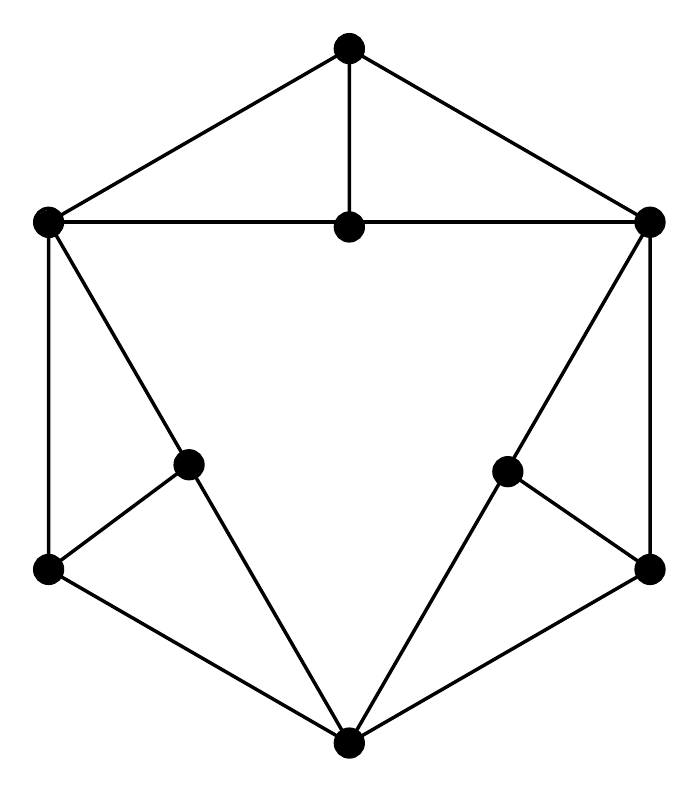}\\
	%\caption{The transmission regular but not degree regular graph with the smallest order.}
	Figure 2: The transmission regular but not degree regular graph with the smallest order.
	\label{fig:tnd}
\end{figure}

The following is straightforward from the definition of status connectivity indices.
\begin{lem}\label{lem:tr-regular}
	Let $G$ be a connected $k$-transmission regular graph with $m$ edges. Then $S_1(G)=2mk$ and $S_2(G)=mk^2.$
\end{lem}

\begin{thm}[\cite{Ashrafi-polyhex}]\label{thm:orbit-w}
	Let $G$ be a connected graph on $n$ vertices with the automorphism group $Aut(G)$ and the vertex set
	$V(G)$. Let $V_1, V_2,\cdots , V_t$ be all orbits of the action $Aut(G)$ on $V(G)$. Suppose that for each
	$1\leq i \leq t$, $k_i$ are the transmission of vertices in the orbit $V_i$,
	respectively. Then
	\[W(G)=\frac{1}{2}\sum_{i=1}^{t} |V_i|k_i.\]
	Specially if $G$ is vertex-transitive (i.e., $t=1$), then $W(G)=\frac{1}{2}nk$, where $k$ is the transmission of each vertex of $G$ respectively.
\end{thm}
Analogous to Theorem \ref{thm:orbit-w} and as a consequence of Proposition \ref{Prop2.1}, we have the following
\begin{thm}\label{thm:orbit}
	Let $G$ be a connected graph on $n$ vertices with the automorphism group $Aut(G)$ and the vertex set
	$V(G)$. Let $V_1, V_2,\cdots , V_t$ be all orbits of the action $Aut(G)$ on $V(G)$. Suppose that for each
	$1\leq i \leq t$, $d_i$ and $k_i$ are the vertex degree and the transmission of vertices in the orbit $V_i$,
	respectively. Then
	\[S_1(G)=\sum_{i=1}^{t} |V_i|d_ik_i, \quad \overline{S_1}(G)=(n-1)\sum_{i=1}^{t} \Big(|V_i|k_i(1-\frac{d_i}{n-1})\Big).\]
	Specially if $G$ is vertex-transitive (i.e., $t=1$), then
	\[S_1(G)=ndk,\quad S_2(G)=\frac{1}{2}ndk^2,\]
	\[\overline{S_1}(G)=2\binom{n}{2}k-ndk,\quad \overline{S_2}(G)=\Big(\binom{n}{2}-\frac{nd}{2}\Big)k^2,\]
	where $d$ and $k$ are the degree and the transmission of each vertex of $G$ respectively.
\end{thm}
%(n-1)nk(1-\frac{d}{n-1})

The following is a direct consequence of Proposition \ref{Prop2.1}, Lemma \ref{lem:tr-regular} and Theorem \ref{thm:orbit-w}.
\begin{lem}\label{lem:tr-regular-coindex}
	Let $G$ be a connected $k$-transmission regular graph with $m$ edges. Then $\overline{S_1}(G)=2\binom{n}{2}k-2mk$ and
	$\overline{S_2}(G)=\binom{n}{2}k^2-mk^2$.
\end{lem}

Following \cite{Harary-book} we recall intersection graphs as follows. Let $S$ be a set and $F=\{S_1,\cdots,S_q\}$ be a
non-empty family of distinct non-empty subsets of $S$ such that $S=\bigcup_{i=1}^{q}S_i$. The intersection
graph of $S$ which is denoted by $\Omega(F)$ has $F$ as its set of vertices and two distinct vertices
$S_i$, $S_j$ , $i\neq j$ , are adjacent if and only if $S_i\bigcap S_j\neq \emptyset$.
Here we will consider a set $S$ of cardinality $p$ and let $F$ be the set of all subsets of
$S$ of cardinality $ t $, $1 < t <p$, which is denoted by $S^{ \{t\} }$. Upon convenience we may set
$S = \{1, 2,\cdots,p\}$. Let us denote the intersection graph $\Omega(S^{\{ t \}})$
by $\Gamma^{ \{t\} } = (V,E)$.
The number of vertices of this graph is $|V|=\binom{p}{ t }$, the
degree $d$ of each vertex is as follows:
\[
d=\left\{
\begin{array}{ll}
\binom{p}{ t }- \binom{p- t }{ t }-1, & \hbox{$p\ge  2t $;} \\[2mm]
\binom{p}{ t }-1, & \hbox{$p< 2t $.}
\end{array}
\right.
\]
The number of its edges is as follows:
\[
|E|=\left\{
\begin{array}{ll}
\frac{1}{2}\binom{p}{ t }(\binom{p}{ t }- \binom{p- t }{ t }-1), & \hbox{$p\ge  2t $;} \\[3mm]
\frac{1}{2}\binom{p}{ t }(\binom{p}{ t }-1), & \hbox{$p< 2t $.}
\end{array}
\right.
\]

\begin{lem}[{\cite{DarafsheHypercube}}]\label{lem:intersect}
	%The automorphism group of $\Gamma^{ \{t\} }$ has a subgroup isomorphic to the symmetric
	%group on n letters.
	%Let $S$ be a set of size $p$. Then
	The intersection graph $\Gamma^{\{t\}}$ is vertex-transitive and for any $t$-element subset $A$ of $S$ we have
	
	\[
	\sigma_{\Gamma^{ \{t\}}}(A)=
	\left\{
	\begin{array}{ll}
	\binom{p}{ t }+\binom{p- t }{ t }-1, & \hbox{$p\ge  2t $;} \\[2mm]
	\binom{p}{ t }-1, & \hbox{$p< 2t $.}
	\end{array}
	\right.
	\]
	Moreover,
	\[
	W(\Gamma^{ \{t\}})=
	\left\{
	\begin{array}{ll}
	\frac{1}{2}\binom{p}{ t }\Big(\binom{p}{ t }+\binom{p- t }{ t }-1\Big), & \hbox{$p\ge  2t $;} \\[2mm]
	\frac{1}{2}\binom{p}{ t }\Big(\binom{p}{ t }-1\Big), & \hbox{$p< 2t $.}
	\end{array}
	\right.
	\]
\end{lem}

\begin{thm}
	\[
	S_1(\Gamma^{ \{t\} })=
	\left\{
	\begin{array}{ll}
	\binom{p}{ t }\Big(\binom{p}{ t }- \binom{p- t }{ t }-1\Big)\Big(\binom{p}{ t }+\binom{p - t }{ t }-1\Big), & \hbox{$p\ge  2t $;} \\[3mm]
	\binom{p}{ t }\Big(\binom{p}{ t }-1\Big)^2, &   \hbox{$p< 2t $.}
	\end{array}
	\right.
	\]
	
	\[
	S_2(\Gamma^{ \{t\} })=
	\left\{
	\begin{array}{ll}
	\frac{1}{2}\binom{p}{ t }\Big(\binom{p}{ t }- \binom{p- t }{ t }-1\Big)\Big(\binom{p}{ t }+\binom{p - t }{ t }-1\Big)^2, & \hbox{$p\ge  2t $;} \\[3mm]
	\frac{1}{2}\binom{p}{ t }\Big(\binom{p}{ t }-1\Big)^3, &   \hbox{$p< 2t $.}
	\end{array}
	\right.
	\]
	\[
	\overline{S_1}(\Gamma^{ \{t\} })=
	\left\{
	\begin{array}{ll}
	\binom{p- t }{ t }\binom{p}{ t }\Big(\binom{p}{ t }+\binom{p - t }{ t }-1\Big), & \hbox{$p\ge  2t $;} \\[3mm]
	2\dbinom{\binom{p}{ t }}{ 2 }(\binom{p}{ t }-1) -\binom{p}{ t }\Big(\binom{p}{ t }-1\Big)^2, &   \hbox{$p< 2t $.}
	\end{array}
	\right.
	\]
	\[
	\overline{S_2}(\Gamma^{ \{t\} })=
	\left\{
	\begin{array}{ll}
	\Big(\dbinom{\binom{p}{ t }}{2}- \frac{1}{2}\binom{p}{ t }(\binom{p}{ t }-\binom{p - t }{ t }-1)\Big)
	\Big( \binom{p}{t}+\binom{p-t}{t}-1\Big)^2 & \hbox{$p\ge  2t $;}\\[3mm]
	%   \binom{p- t }{ t }\binom{p}{ t }\Big(\binom{p}{ t }+\binom{p - t }{ t }-1\Big),
	\Big(\dbinom{\binom{p}{t}}{2}- \frac{1}{2}\binom{p}{t}(\binom{p}{ t }-1)\Big)
	\Big( \binom{p}{t}-1\Big)^2, &   \hbox{$p< 2t $.}
	\end{array}
	\right.
	\]
\end{thm}
\begin{proof}
	It is a direct consequence of Theorem \ref{thm:orbit} and Lemma \ref{lem:intersect}.
\end{proof}

The vertex set of the hypercube $H_n$ consists of all $n$-tuples $(b_1,b_2,\cdots,b_n)$ with
$b_i \in \{0, 1\}$. Two vertices are adjacent if the corresponding tuples differ in precisely
one place. Moreover, $H_n$ has exactly $2n$
vertices and $n2^{n-1}$ edges. Darafsheh \cite{DarafsheHypercube} proved that $H_n$ is vertex-transitive  and for every vertex $u$, $\sigma_{H_n}(u) = n2^{n-1}$. Therefore, by Lemmas \ref{lem:tr-regular} and \ref{lem:tr-regular-coindex} we have following result.

\begin{thm}
	For hypercube $H_{n}$
	\begin{equation}\label{eq:hypercube}\nonumber
	S_1(G)=n^22^{2n-1} ~~\text{and}~~ S_2(G)=n^32^{3n-3},
	\end{equation}
	
	\begin{equation}\label{eq:hypercube-co}\nonumber
	\overline{S_1}(G)=2n^22^{n-1}(2n-5)~~\text{and}~~ \overline{S_2}(G)=n^22^{2n-2}(n(2n-1)-1).
	\end{equation}
\end{thm}

The \textit{Kneser graph} $KG_{p,k}$ is the graph whose vertices correspond to the $k$-element subsets of
a set of $p$ elements, and where two vertices are adjacent if and only if the two corresponding sets are
disjoint. Clearly we must impose the restriction $p \ge 2k$. The Kneser graph $KG_{p,k}$ has
$\binom{p}{k}$ vertices and it is regular of degree $\binom{p-k}{k}$. Therefore the number of edges
of $KG_{p,k}$ is $\frac{1}{2}\binom{p}{k}\binom{p-k}{k}$
(see \cite{darafshKneser}).
The Kneser graph $KG_{n,1}$ is the \textit{complete graph} on $n$ vertices. The Kneser graph $KG_{2p-1,p-1}$ is known as
the \textit{odd graph} $O_p$. The odd graph $O_3 = KG_{5,2}$ is isomorphic to the \textit{Petersen graph} (see Figure 3).

\begin{figure}[h!]\label{fig:peterson}
	\centering
	\includegraphics[width=4.6cm]{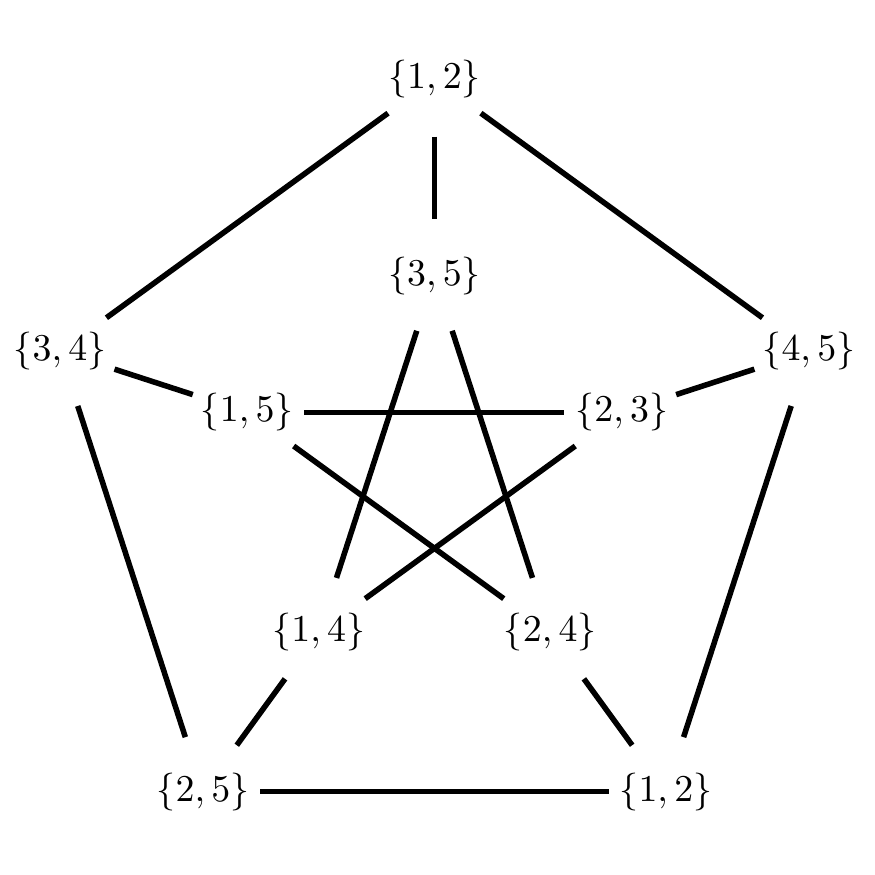}\\
	%\caption{The odd graph $O_3 = KG_{5,2}$ is isomorphic to the \textit{Petersen graph}}
	Figure 3: The odd graph $O_3 = KG_{5,2}$ is isomorphic to the \textit{Petersen graph}
		\label{}
\end{figure}

\begin{lem}[{\protect{\cite{darafshKneser}}}]\label{lem:Kneser}%[Lemma 3.3.]
	The Kneser graph $KG_{p,k}$ is vertex-transitive and for each $k$-subset $A$, $\sigma_{KG_{p,k}}(A)=\frac{2W(KG_{p,k})}{\binom{p}{k}}$.
\end{lem}

Following Proposition follows from Lemma \ref{lem:Kneser} and Lemma \ref{lem:tr-regular}.

\begin{prop}\label{Prop3.4}
	For a Kneser graph $KG_{p,k}$ we have
	\[S_1(KG_{p,k})=2W(KG_{p,k})\binom{p-k}{k} \]
	and \[S_2(KG_{p,k})=\binom{p-k}{k}\left[\frac{2(W(KG_{p,k}))^2}{\binom{p}{k}}\right].\]
\end{prop}

Following Proposition follows from Proposition \ref{Prop2.1}, Lemma \ref{lem:Kneser} and Proposition \ref{Prop3.4}.

\begin{prop}\label{Prop3.5}
	For a Kneser graph $KG_{p,k}$ we have
	\[\overline{S_1}(KG_{p,k})=2W(KG_{p,k})\left[\binom{p}{k} - \binom{p-k}{k} - 1 \right] \]
	and
	\[ \overline{S_2}(KG_{p,k})=2(W(KG_{p,k}))^2 - W(KG_{p,k}) - \binom{p-k}{k}\left[\frac{2(W(KG_{p,k}))^2}{\binom{p}{k}}\right].\]
\end{prop}

A nanostructure called \textit{achiral polyhex nanotorus} (or \textit{toroidal fullerenes} of perimeter $p$ and length $q$, denoted by $T[p, q]$ is depicted in Figure 4 and its 2-dimensional molecular graph is in Figure 5. It is regular of degree 3 and has $pq$ vertices and $\frac{3pq}{2}$ edges.

\begin{figure}[h!]
	\centering
	\includegraphics[width=4.5cm]{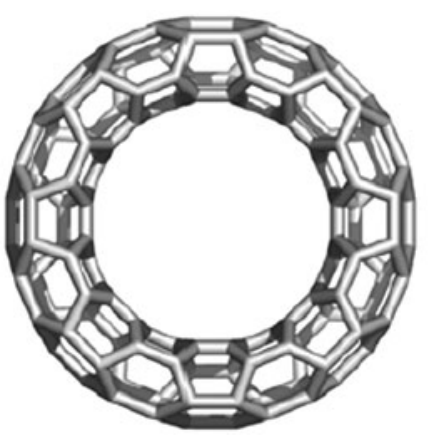}\\
	%\caption{A achiral polyhex nanotorus (or toroidal fullerene) $T[p,q]$}
	Figure 4: A achiral polyhex nanotorus (or toroidal fullerene) $T[p,q]$
	\label{fig:apn}
\end{figure}

\begin{figure}[h!]
	\centering
	\includegraphics[width=4.5cm]{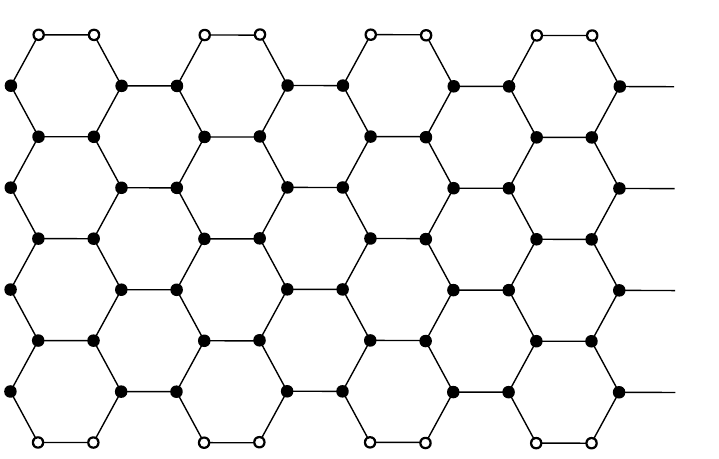}\\
	%\caption{A 2-dimensional lattice for an achiral polyhex nanotorus $T[p,q]$}
	Figure 5: A 2-dimensional lattice for an achiral polyhex nanotorus $T[p,q]$
	\label{fig:apnlattice}
\end{figure}

The following lemma was proved in \cite{Ashrafi-polyhex} and \cite{Yousefi-2008}.
\begin{lem}[\cite{Ashrafi-polyhex},\cite{Yousefi-2008}]\label{lem:pn-vt}
	The achiral polyhex nanotorus $T = T[p,q]$ is vertex transitive such that for an arbitrary vertex $u\in V(T)$
	\[
	\sigma_T(u)=\left\{
	\begin{array}{ll}
	\dfrac{q}{12}(6p^2+q^2-4), & \hbox{$q<p$;} \\[5mm]
	\dfrac{p}{12}(3q^2+3pq+p^2-4), & \hbox{$q\ge p$.}
	\end{array}
	\right.\]
\end{lem}
The following is a direct consequence of Lemma \ref{lem:tr-regular} and Lemma \ref{lem:pn-vt}.
%m=3\frac{pq}{2}
\begin{cor} \label{Cor3.7}
	Let $T = T[p,q]$ be a achiral polyhex nanotorus. Then
	\[
	S_1(T)=\left\{
	\begin{array}{ll}
	\dfrac{pq^2}{4}(6p^2+q^2-4),     & \hbox{$q<p$;} \\[5mm]
	\dfrac{p^2q}{4}(3q^2+3pq+p^2-4), & \hbox{$q\ge p$.}
	\end{array}
	\right.\]
	And
	\[
	S_2(T)=\left\{
	\begin{array}{ll}
	\dfrac{pq^3}{96}(6p^2+q^2-4)^2,     & \hbox{$q<p$;} \\[5mm]
	\dfrac{p^3q}{96}(3q^2+3pq+p^2-4)^2, & \hbox{$q\ge p$.}
	\end{array}
	\right.\]
\end{cor}

\begin{cor} \label{Cor3.8}
	Let $T = T[p,q]$ be a achiral polyhex nanotorus. Then
	\[
	\overline{S_1}(T)=\left\{
	\begin{array}{ll}
	\dfrac{pq^2}{12}(pq-4)(6p^2+q^2-4),     & \hbox{$q<p$;} \\[5mm]
	\dfrac{p^2q}{12}(pq-4)(3q^2+3pq+p^2-4), & \hbox{$q\ge p$.}
	\end{array}
	\right.\]
	And
	\[
	\overline{S_2}(T)=\left\{
	\begin{array}{ll}
	\dfrac{pq^3}{288}(pq-4)(6p^2+q^2-4)^2,     & \hbox{$q<p$;} \\[5mm]
	\dfrac{p^3q}{288}(pq-4)(3q^2+3pq+p^2-4)^2, & \hbox{$q\ge p$.}
	\end{array}
	\right.\]
\end{cor}

\begin{proof}
	Since $2W(G) = \sum_{u \in V(G)}\sigma_G(u)$ and polyhex nanotorus $T[p, q]$ has $pq$ vertices, by Lemma \ref{lem:pn-vt}, the Wiener index of polyhex nanotorus $T[p, q]$ is as follows \cite{Yousefi-2008}:
	
	\[
	W(T)=\left\{
	\begin{array}{ll}
	\dfrac{pq^2}{24}(6p^2+q^2-4),     & \hbox{$q<p$;} \\[5mm]
	\dfrac{p^2q}{24}(3q^2+3pq+p^2-4), & \hbox{$q\ge p$.}
	\end{array}
	\right.\]
	
	Substituting this in the Proposition \ref{Prop2.1} and Corollary \ref{Cor3.7} we get the results.
\end{proof}

\noindent
\textbf{Acknowledgement:}
The author HSR is thankful to the University Grants Commission (UGC), Govt. of India for support through grant under
UGC-SAP DRS-III, 2016-2021: F.510/3/DRS-III /2016 (SAP-I). The second author
ASY is thankful to the University Grants Commission (UGC), Govt. of India for support through Rajiv Gandhi National Fellowship No. F1-17.1/2014-15/RGNF-2014-15-SC-KAR-74909.

\end{document}